\documentclass[11pt,a4paper]{article}
 \usepackage{euscript}
\usepackage{amsmath}
\usepackage{graphicx}
\usepackage{amsthm}
\usepackage{amssymb}
\usepackage{epsfig}
\usepackage{url}
\usepackage{amsmath,amscd}
\theoremstyle{theorem}
\newtheorem{theorem}{Theorem}[section]
\newtheorem{corollary}[theorem]{Corollary}

\newtheorem{lemma}[theorem]{Lemma}
\newtheorem{proposition}[theorem]{Proposition}

\numberwithin{equation}{section}
 
\title{Rational twisted  series} 
\author{Masood Aryapoor\\
\tiny{\textit{Division of Mathematics and Physics}}\\
\tiny{\textit{M\"{a}lardalen  University}}\\
\tiny{\textit{Hamngatan 15, 632 17, Eskilstuna, 
Sweden
}}

}
 \date{}
\begin{document}
 \maketitle
\begin{abstract}
\noindent
 Rational twisted power series over a (commutative) filed are studied. We give several characterizations of such series which are similar to the classical results concerning rational power series over a commutative field. In particular, we prove a version of Kronecker's lemma for the rationality of twisted power series.    
\end{abstract}
%%%%%%%%%%%%%%%%%%%%%%%%%%%%%%%%%%%%%%%%%%%%%%%%%%%%%%%%%%%%%%%%
%%%%%%%%%%%%%%%%%%%%%%%%%%%%%%%%%%%%%%%%%%%%%%%%%%%%%%%%%%%%%%%%
%%%%%%%%%%%%%%%%%%%%%%%%%%%%%%%%%%%%%%%%%%%%%%%%%%%%%%%%%%%%%%%%
%%%%%%%%%%%%%%%%%%%%%%%%%%%%%%%%%%%%%%%%%%%%%%%%%%%%%%%%%%%%%%%%
\begin{section}{Introduction} 

The study of  rational formal power series has led to a rich theory with applications in areas such as number theory and formal language theory, see  \cite{Salem_1983} and \cite{Berstel-Reutenauer_1988}. The main purpose of this paper is to develop a theory of rational twisted formal power series over a commutative field, similar to the theory of rational formal power series as presented in \cite{Berstel-Reutenauer_1988}.

The paper is organized as follows. Section 2 is devoted to preliminaries. Most of the results in this section are classical. The only new result is Kronecker's lemma for rational twisted formal power series, see Subsection \ref{Kroneckerlemma}. In Section 3, the concept of $\sigma$-recognizable formal power series is introduced and studied. The most important result of this section is that a twisted formal power series is $\sigma$-recognizable iff it is rational. It is also shown that the Hadamard product of rational twisted formal power series is rational. Section 4 deals with the notion of syntactic left ideal which is used to give a characterization of rational twisted formal power series. More precisely, it is shown that a twisted formal power series is rational iff its syntactic left ideal is nonzero. The next section is devoted to reduced  $\sigma$-linear representations. In the final section of the paper, we investigate the connection between fractions of the twisted polynomial ring and syntactic left ideals. Also, the notion of regularity is introduced and it is shown that the ableian group of regular twisted formal power series equipped with the Hadamard product is a commutative algebra.   

Throughout this paper, $K$ is  a fixed commutative field, and $\sigma:K\to K$ is a fixed endomorphism of $K$. By "vector space", we shall mean "vector space over $K$". The notation $K^{m\times n}$ is used to denote the vector space of all $m\times n$ matrices with entries in $K$. 
\end{section} 
%%%%%%%%%%%%%%%%%%%%%%%%%%%%%%%%%%%%%%%%%%%%%%%%%%%%%%%%%%%%%%%%
%%%%%%%%%%%%%%%%%%%%%%%%%%%%%%%%%%%%%%%%%%%%%%%%%%%%%%%%%%%%%%%%

%%%%%%%%%%%%%%%%%%%%%%%%%%%%%%%%%%%%%%%%%%%%%%%%%%%%%%%%%%%%%%%%
%%%%%%%%%%%%%%%%%%%%%%%%%%%%%%%%%%%%%%%%%%%%%%%%%%%%%%%%%%%%%%%%
%%%%%%%%%%%%%%%%%%%%%%%%%%%%%%%%%%%%%%%%%%%%%%%%%%%%%%%%%%%%%%%%
%%%%%%%%%%%%%%%%%%%%%%%%%%%%%%%%%%%%%%%%%%%%%%%%%%%%%%%%%%%%%%%%

%%%%%%%%%%%%%%%%%%%%%%%%%%%%%%%%%%%%%%%%%%%%%%%%%%%%%%%%%%%%%%%%
%%%%%%%%%%%%%%%%%%%%%%%%%%%%%%%%%%%%%%%%%%%%%%%%%%%%%%%%%%%%%%%%
%%%%%%%%%%%%%%%%%%%%%%%%%%%%%%%%%%%%%%%%%%%%%%%%%%%%%%%%%%%%%%%%
%%%%%%%%%%%%%%%%%%%%%%%%%%%%%%%%%%%%%%%%%%%%%%%%%%%%%%%%%%%%%%%%
%%%%%%%%%%%%%%%%%%%%%%%%%%%%%%%%%%%%%%%%%%%%%%%%%%%%%%%%%%%%%%%%
%%%%%%%%%%%%%%%%%%%%%%%%%%%%%%%%%%%%%%%%%%%%%%%%%%%%%%%%%%%%%%%%

%%%%%%%%%%%%%%%%%%%%%%%%%%%%%%%%%%%%%%%%%%%%%%%%%%%%%%%%%%%%%%%%
%%%%%%%%%%%%%%%%%%%%%%%%%%%%%%%%%%%%%%%%%%%%%%%%%%%%%%%%%%%%%%%%
%%%%%%%%%%%%%%%%%%%%%%%%%%%%%%%%%%%%%%%%%%%%%%%%%%%%%%%%%%%%%%%%
%%%%%%%%%%%%%%%%%%%%%%%%%%%%%%%%%%%%%%%%%%%%%%%%%%%%%%%%%%%%%%%%
\begin{section}{Twisted series}
	In this section, we review basic facts about twisted polynomial rings and twisted formal power series, see \cite{Cohn_2006} for more details.  
	%%%%%%%%%%%%%%%%%%%%%%%%%%%%%%%%%%%%%%%%%%%%%%%%%%%%%%%%%%%%%%%%% 
	\begin{subsection}{The ring $K[T;\sigma]$} 	
		Every nonzero element $P$ of the ring $K[T;\sigma]$ of twisted polynomials over $K$ can uniquely be written as 
		$$P=a_0+a_1T+\cdots+a_n T^n=\sum_{i=0}^n a_i T^i,$$
		where $n\geq 0$ and $a_0,...,a_n\in K$ with $a_n\neq 0$. The number $n$, denoted by  $\deg P$, is called the degree of $P$. The notation $P(0)$ will be used to denote the constant term $a_0$. In this ring, we have $Ta=\sigma(a)T$ for all $a\in K$, and more generally, $$\left( \sum_{i}a_i T^i \right) \left( \sum_{j}b_j T^j \right)=
		 \sum_{n}\left( \sum_{i+j=n}a_i\sigma^i(b_j)\right)  T^n.$$
		
	\end{subsection}
	%%%%%%%%%%%%%%%%%%%%%%%%%%%%%%%%%%%%%%%%%%%%%%%%%%%%%%%%%%%%%%%%% 
	\begin{subsection}{The field of fractions of $K[T;\sigma]$} 	
		It is known that $K[T;\sigma]$ is an integral domain. Moreover, the Euclidean algorithm holds for right division, that is, given $P,Q\in K[T;\sigma]$ with $Q\neq 0$ there exist $S,R\in K[T;\sigma]$ such that $\deg R<\deg Q$ and $P=SQ+R$. It follows that $K[T;\sigma]$ is a left principal ideal domain. Since    $K[T;\sigma]$  is left Noetherian, we see that $K[T;\sigma]$  is a left Ore domain, and consequently, has a field of fractions which is denoted by $K(T;\sigma)$. Any element of $K(T;\sigma)$ can be written as a left fraction $P^{-1}Q$ where $P,Q\in  K[T;\sigma]$ with $P\neq 0$. 	We remark that if $\sigma$ is an automorphim, then $K[T;\sigma]$ is also a right principal ideal domain and a right Ore domain.  	
	\end{subsection}

	%%%%%%%%%%%%%%%%%%%%%%%%%%%%%%%%%%%%%%%%%%%%%%%%%%%%%%%%%%%%%%%%% 
	\begin{subsection}{The ring of twisted formal power series} 	
		Every element $f$ of the ring $K[[T;\sigma]]$ of twisted formal series over $K$ can uniquely be written as 
		$$f=\sum_{i=0}^\infty a_i T^i=a_0+a_1T+\cdots+a_n T^n+\cdots,$$
		where  $a_0,a_1,...\in K$. For simplicity, elements of $K[[T;\sigma]]$ will be called \textit{twisted series}. As before, $f(0)$ denotes the constant term $a_0$. The ring $K[[T;\sigma]]$ can be considered as a topological ring in which the left ideals $K[[T;\sigma]]T^n$, where $n\geq 0$, form a neighborhood basis of $0$. Note that an element $f\in K[[T;\sigma]]$ has an inverse $f^{-1}\in K[[T;\sigma]]$ iff $f(0)\neq 0$. In fact, if $g\in  K[[T;\sigma]]$  satisfies $g(0)=0$, then 
		$$(1-g)^{-1}=\sum_{i=0}^{\infty }g^i.$$		
		We also have the ring of twisted formal Laurent series  	$K((T;\sigma))$ which contains 	 $K[[T;\sigma]]$  as a subring.  Every nonzero element $f$ of the ring $K((T;\sigma))$ can uniquely be written as 
		$$f=\sum_{i=-m}^\infty a_i T^i=a_{-m}T^{-m}+\cdots+a_{-1}T^{-1}+a_0+a_1T+\cdots,$$
		where $m\in\mathbb{Z}$ and $a_{-m}, a_{-m+1},...\in K$ with $a_{-m}\neq 0$. It is known that $K((T;\sigma))$ is a (skew) field and we have a natural embedding of $K(T;\sigma)$ in $K((T;\sigma))$. In this paper, we consider  $K(T;\sigma)$ as a subfield of $K((T;\sigma))$ under this natural embedding.  
	\end{subsection}
	%%%%%%%%%%%%%%%%%%%%%%%%%%%%%%%%%%%%%%%%%%%%%%%%%%%%%%%%%%%%%%%%% 
	\begin{subsection}{Rational twisted series} 	
		A twisted series  $f\in K[[T;\sigma]]$ is called rational if $f\in K(T;\sigma)$, that is, there exist twisted polynomials $P,Q\in K[T;\sigma]$ such that $P\neq 0$ and 
		$$f=P^{-1}Q.$$
		We give a proof of the following known rationality criterion, see Theorem 1.5.6 in \cite{Cohn_2006}.  
		\begin{theorem}\label{(T)rationalitycriterion1}
			A twisted series $$f=\sum_{i=0}^\infty a_iT^i\in K[[T;\sigma]],$$  is  rational iff
			there exist elements $c_1,...,c_p\in K$ such that 	$$a_{n}+\sum_{i=1}^{p}c_{i}\sigma^{i}(a_{
				n-i})=0,$$
			for all  $n\geq n_0$. 
		\end{theorem}
	\begin{proof}
		The proof of the "if" part is easy since 
		$$(1+\sum_{i=1}^{p}c_iT^{i})f\in K[T;\sigma].$$
		To prove the other direction, let $f=\sum_{i=0}^\infty a_iT^i$ be rational. Then, there exists a twisted polynomial $$P=b_0T^k+\cdots+b_{p}T^{p+k},$$
		where $p,k \geq 0$ and $b_0,...,b_{p}\in K$, such that
		$$Pf\in K[T;\sigma].$$
		Since  
		$$Pf=(b_0T^k+\cdots+b_{p}T^{p+k})(\sum_{i=0}^\infty a_iT^i)=\sum_{n\geq 0} \left( \sum_{i+j=n}b_i\sigma^{k+i}(a_{j})\right) T^{k+n},$$
		we obtain 
		$$\sum_{i=0}^{p}b_i\sigma^{k+i}(a_{n-i})=0,$$
		for all $n\geq n_0\geq p$. Therefore, the vectors 
		$$v_i=(\sigma^{k+i}(a_{n_0-i}),\sigma^{k+i}(a_{n_0+1-i}),...,\sigma^{k+i}(a_{n-i}),... ),$$
		where $i=0,...,p$, are linearly dependent over $K$. Without loss of generality, we may assume that $v_1,...,v_{p}$ are linearly independent over $K$.
		Setting
		 $$w_i=(\sigma^{i}(a_{n_0-i}),\sigma^{i}(a_{n_0+1-i}),...,\sigma^{i}(a_{n-i}),... ),$$
		we have $v_i=\sigma^k(w_i)$, for $i=0,...,p$. Since $\sigma$ is a field homomorphism, it follows that $w_1,...,w_{p}$ are linearly independent, and $w_0,...,w_{p}$ are linearly dependent over $K$. Therefore, there exist elements $c_1,...,c_{p}\in K$ such that $w_0+c_1w_1+\cdots+c_{p}w_{p}=0$, or equivalently,
		$$a_{n}+\sum_{i=1}^{p}c_i\sigma^{i}(a_{n-i})=0,$$
		for all $n\geq n_0$, finishing the proof.  
	\end{proof}
	The proof of the theorem gives us the following result:
	\begin{corollary}\label{denominator1}
		Any rational twisted series $f\in K[[T;\sigma]]$ can be written as $f=P^{-1}Q$
		for some $P,Q\in K[T;\sigma]$ such that $P(0)=1$. 
	\end{corollary}
	
	\end{subsection}
%%%%%%%%%%%%%%%%%%%%%%%%%%%%%%%%%%%%%%%%%%%%%%%%%%%%%%%%%%%%%%%%%
\begin{subsection}{Kronecker's lemma for twisted series} \label{Kroneckerlemma}	
	In this part, we prove a criterion for the rationality of a twisted  series which is similar to a classical result for ordinary power series known as Kronecker's lemma, see \cite{Salem_1983}.

	\begin{lemma}
		A twisted series $$f=\sum_{i=0}^\infty a_iT^i\in K[[T;\sigma]],$$  is  rational	iff  
		$$\begin{vmatrix}
			\sigma^m(a_0)&\sigma^{m-1}(a_1)&\cdots&\sigma(a_{m-1})&a_m\\
			\sigma^m(a_1)&\sigma^{m-1}(a_2)&\cdots&\sigma(a_{m})&a_{m+1}\\
			\vdots&\vdots&\cdots&\vdots&\vdots\\
			\sigma^m(a_m)&\sigma^{m-1}(a_{m+1})&\cdots&\sigma(a_{2m-1})&a_{2m}
		\end{vmatrix}=0
		$$ for all $m>m_0$.
	\end{lemma}
	\begin{proof}
		Set
		$$D_m=\begin{pmatrix}
			\sigma^m(a_0)&\sigma^{m-1}(a_1)&\cdots&\sigma(a_{m-1})&a_m\\
			\sigma^m(a_1)&\sigma^{m-1}(a_2)&\cdots&\sigma(a_{m})&a_{m+1}\\
			\vdots&\vdots&\cdots&\vdots&\vdots\\
			\sigma^m(a_m)&\sigma^{m-1}(a_{m+1})&\cdots&\sigma(a_{2m-1})&a_{2m}\end{pmatrix},$$
		where $m\geq 0$. 
		First suppose that $f$ is rational. By Theorem \ref{(T)rationalitycriterion1}, there are $c_1,...,c_p\in K$ such that 	$$a_{n}+\sum_{i=1}^{p}c_{i}\sigma^{i}(a_{
			n-i})=0,$$
		for all $n\geq n_0$. It follows that the determinant of $D_m$ is zero for all  $m\geq n_0$.
		
		Conversely, suppose that $|D_m|=0$ for all  $m>n_0$. If $\left| D_m\right| =0$ for all $m=0,1,2,...$, then it is easy to see that we must have $f=0$, and we are done. Otherwise, we can find $p\in\mathbb{N}$ such that $|D_{p-1}|\neq 0$ and $|D_m|=0$ for all $m\geq p$. Since $\left| D_{p}\right|=0$ and $\left| D_{p-1}\right|\neq 0$, we can find 	$c_1,...,c_{p}\in K$ such that
		$$c_p\sigma^p(a_j)+c_{p-1}\sigma^{p-1}(a_{j+1})+\cdots+c_{1}\sigma(a_{j+p-1})+a_{j+p}=0,$$
		where $0\leq j\leq p$. We set
		$$L_{j}=a_{j+p}+ \sum_{k=1}^pc_{k}\sigma^{k}(a_{j+p-k}),$$
		and use induction to prove that $L_{j}=0$ for all $j\geq 0$. Assume that $m>p$ and $L_{j}=0$ when $j<m$. We note that 
		$$\sigma^r(L_{j})=\sigma^r(a_{j+p})+ \sum_{k=1}^p\sigma^r(c_{k})\sigma^{k+r}(a_{j+p-k}),$$
%		\begin{align*}
%			\sigma^k(L_{j}) &=\\
%			&= \sigma^k(c_p)\sigma^{k+p}(a_j)+\sigma^k(c_{p-1})\sigma^{k+p-1}(a_{j+1})+\cdots+\sigma^k(c_{1})\sigma^{k+1}(a_{j+p-1})+\sigma^{k}(a_{j+p}),
%		\end{align*}
		for all $r\geq 0$, since $\sigma$ is a field homomorphism. Now, let $A=(c_{ij})$ be the following $(m+1)\times (m+1)$ matrix 
		$$c_{ij}=
		\begin{cases}\
			\sigma^{m+1-j}(c_{j-i}) & j> p\text{ and } j-p\leq i\leq j-1,\\
			1 & i=j,\\
			0 & \text{otherwise.}\\
		\end{cases}
		$$
		Then the $ij$-th entry of $D_mA$ is equal to $\sigma^{m-j+1}(a_{i+j-1})$ if $j\leq p$, and equal to
		$$
		\sigma^{m-j+1}(a_{i+j-1})+\sum_{k=j-p}^{j-1} \sigma^{m-k+1}(a_{i+k-1})\sigma^{m+1-j}(c_{j-k})=
		$$
		\begin{align*}
			\sigma^{m-j+1}\left( a_{i+j-1}+\sum_{k=j-p}^{j-1}c_{j-k} \sigma^{j-k}(a_{i+k-1})\right) &=\sigma^{m-j+1}\left( a_{i+j-1}+\sum_{k=1}^{p}c_{k} \sigma^{k}(a_{i+j-1-k})\right)\\
			&= \sigma^{m-j+1}\left( L_{i+j-p-1}\right),
		\end{align*}				
		if $j> p$. Since, $|A|=1$, and $L_{j}=0$ for $j<m$, we see that 
		$$|D_m|=(-1)^{m-p}|D_{p-1}| \prod_{j=p+1}^{m}\sigma^{m-j+1}\left( L_{m}\right).$$
		It follows that $L_m=0$ since $|D_m|=0$, $|D_{p-1}|\neq 0$, and $\sigma$ is injective. The proof of the induction step is finishes. Since $L_j=0$ for all $j\geq 0$, by Lemma \ref{(T)rationalitycriterion1}, $f$ must be rational, and we are done.  
		
	\end{proof}
	
\end{subsection}
	%%%%%%%%%%%%%%%%%%%%%%%%%%%%%%%%%%%%%%%%%%%%%%%%%%%%%%%%%%%%%%%%% 

	%%%%%%%%%%%%%%%%%%%%%%%%%%%%%%%%%%%%%%%%%%%%%%%%%%%%%%%%%%%%%%%%
	%%%%%%%%%%%%%%%%%%%%%%%%%%%%%%%%%%%%%%%%%%%%%%%%%%%%%%%%%%%%%%%%
\end{section} 
%%%%%%%%%%%%%%%%%%%%%%%%%%%%%%%%%%%%%%%%%%%%%%%%%%%%%%%%%%%%%%%%
%%%%%%%%%%%%%%%%%%%%%%%%%%%%%%%%%%%%%%%%%%%%%%%%%%%%%%%%%%%%%%%%
%%%%%%%%%%%%%%%%%%%%%%%%%%%%%%%%%%%%%%%%%%%%%%%%%%%%%%%%%%%%%%%%
%%%%%%%%%%%%%%%%%%%%%%%%%%%%%%%%%%%%%%%%%%%%%%%%%%%%%%%%%%%%%%%%
%%%%%%%%%%%%%%%%%%%%%%%%%%%%%%%%%%%%%%%%%%%%%%%%%%%%%%%%%%%%%%%%
%%%%%%%%%%%%%%%%%%%%%%%%%%%%%%%%%%%%%%%%%%%%%%%%%%%%%%%%%%%%%%%%
%%%%%%%%%%%%%%%%%%%%%%%%%%%%%%%%%%%%%%%%%%%%%%%%%%%%%%%%%%%%%%%%
%%%%%%%%%%%%%%%%%%%%%%%%%%%%%%%%%%%%%%%%%%%%%%%%%%%%%%%%%%%%%%%%
%%%%%%%%%%%%%%%%%%%%%%%%%%%%%%%%%%%%%%%%%%%%%%%%%%%%%%%%%%%%%%%%
%%%%%%%%%%%%%%%%%%%%%%%%%%%%%%%%%%%%%%%%%%%%%%%%%%%%%%%%%%%%%%%%
\begin{section}{$\sigma$-Recognizable twisted series} 

A twisted series $f=\sum_{n\geq 0}a_nT^n\in K[[T;\sigma]]$ is called $\sigma$\textit{-recognizable} if there exist a matrix $A\in K^{m\times m}$, a row vector $X\in K^{1\times m}$ and a column vector $Y\in K^{m\times 1}$ such that 
\begin{equation}\label{a_n}
 a_n=X\left( A\sigma(A)\cdots\sigma^{n-1}(A)\right) \sigma^{n} (Y),
\end{equation}
for all $n\geq 0$ (note that for $n=0$, the formula reads $a_0=XY$). The triple $(X,A,Y)$ is called a $\sigma$-\textit{linear representation} of $f$, and the number $m$ is called the \textit{dimension} of $(X,A,Y)$. We note that Equation \ref{a_n} can be written as
$$f=X(I_m-AT)^{-1}Y,$$
where $$(I_m-AT)^{-1}=\sum_{n\geq 0}(A T)^n=\sum_{n\geq 0}A\sigma(A)\cdots\sigma^{n-1}(A) T^n\in K[[T;\sigma]]^{m\times m}.$$

In order to study $\sigma$-recognizable twisted series, we use the shift map  $s:K[[T;\sigma]]\to K[[T;\sigma]]$ which is defined as follows
	$$s(\sum_{n\geq 0}a_nT^n)=\sum_{n\geq 0}a_{n+1}T^n.$$
	The proof of the following lemma is left  to the reader. 
	\begin{lemma}\label{(L)Propertiesofs}
	(1) $s(f+g)=s(f)+s(g)$ for all $f,g\in K[[T;\sigma]]$.\\
	(2) $s(af)=as(f)$  for all $f\in K[[T;\sigma]]$ and $a\in K$.\\
	(3) $s(fa)=s(f)\sigma(a)$  for all $f\in K[[T;\sigma]]$ and $a\in K$.\\
	(4) If $f=\sum_{n\geq 0}a_nT^n\in K[[T;\sigma]]$, then $a_n=s^n(f)(0)$ for all $n\geq 0$. 	
	\end{lemma}

%%%%%%%%%%%%%%%%%%%%%%%%%%%%%%%%%%%%%%%%%%%%%%%%%%%%%%%%%%%%%%%%%
%%%%%%%%%%%%%%%%%%%%%%%%%%%%%%%%%%%%%%%%%%%%%%%%%%%%%%%%%%%%%%%%% 
	\begin{subsection}{Stability}\label{(S)Stability} 	
	A subset $P$ of $K[[T;\sigma]]$ is called \textit{stable} if $s(P)\subset P$. 
	\begin{proposition}\label{(P)RightStableRecog}
		A twisted series is $\sigma$-recognizable iff it belongs to a finite-dimensional stable  right vector subspace of $K[[T;\sigma]]$.  More precisely,  the following hold:\\
		(i) Let $f_1,...,f_m$ be a basis of 
		a stable  right vector subspace $V$ of $K[[T;\sigma]]$ containing $f$. Setting $F=\begin{pmatrix}
			f_1&
			f_2&
			\cdots&
			f_m
		\end{pmatrix} \in K[[T;\sigma]]^{1\times m},$  we have the following $\sigma$-linear representation of $f$
	$$(F(0),A,Y),$$
	where $Y\in K^{m\times 1}$ and $A\in K^{m\times m}$ satisfy $f=FY$ and $s(F)=FA$. \\
	(ii) Let $(X,A,Y)$ be a $\sigma$-linear representation of $f$. The right vector subspace of  $K[[T;\sigma]]$ generated by the components of the vector $$X(I_m-AT)^{-1}\in K[[T;\sigma]]^{1\times m},$$ is a finite-dimensional stable right vector subspace of $K[[T;\sigma]]$ containing $f$.
	\end{proposition}
	\begin{proof}
		(i) 
		We can find a vector $Y\in K^{m\times 1}$ such that
		$f=FY$ since $f\in V$. The fact that $V$ is stable, implies that there is a matrix $A\in K^{m\times m}$ such that  $s(F)=FA$. Using Part (3) of  Lemma \ref{(L)Propertiesofs} and induction on $n$, one can show that 
		$$s^{n}(f)=\left( FA\sigma(A)\cdots\sigma^{n-1}(A)\right) \sigma^{n}( Y),$$
		for all $n\geq 0$. It follows that
		$$a_n=s^{n}(f)(0)=F(0)\left( A\sigma(A)\cdots\sigma^{n-1}(A)\right) \sigma^{n} (Y) ,$$
		for all $n\geq 0$, proving that $(F(0),A,Y)$ is a $\sigma$-linear representation of $f$.
		\newline
		\newline
		(ii) Clearly, $V$ is finite-dimensional, and $f\in V$ because $f=X(I_m-AT)^{-1}Y$. We have
		\begin{align*}
			s\left( X(I_m-AT)^{-1} \right) &=s\left( \sum_{n\geq 0}  \left( X A\sigma(A)\cdots\sigma^{n-1}(A)\right)T^n \right) \\
			&=\sum_{n\geq 0}  \left( X A\sigma(A)\cdots\sigma^{n}(A)\right) T^n\\
			&=\sum_{n\geq 0}  \left( X A\sigma(A)\cdots\sigma^{n-1}(A)\right) T^nA=FA,
		\end{align*}
		showing that $V$ is stable, and we are done. 
	\end{proof}
	We note that if $(X,A,Y)$ is a $\sigma$-linear representation of $f\in  K[T;\sigma]$, then $(aX,A,Y)$ is a $\sigma$-linear representation of $af$, and $(X,A,aY)$ is a $\sigma$-linear representation of $fa$ for all $a\in K$. Since the sum of stable sets is stable, we obtain the following result:
	\begin{corollary}\label{(C)Sumofrec}
		Any (left or right) $K$-linear combination of  $\sigma$-recognizable twisted series is  $\sigma$-recognizable. 
	\end{corollary}
	
	Clearly, the intersection of any family of stable right vector subspaces of  $ K[[T;\sigma]]$ is a stable right vector subspace of  $ K[[T;\sigma]]$. Therefore, we can define the concept of the smallest stable right vector subspace $V(S)$ of $ K[[T;\sigma]]$ generated by a set $S\subset K[[T;\sigma]]$. It is easy to see that the vector space 
	$$V(f)=\sum_{i\geq 0}s^i(f)K,$$ is the smallest stable right vector subspace of $ K[[T;\sigma]]$ which contains $f\in  K[[T;\sigma]]$. Therefore, $f$ is $\sigma$-recognizable iff $V(f)$ is finite-dimensional. Furthermore, we have the following result:
	\begin{proposition}\label{(P)smalleststable}
		A twisted series $f\in  K[T;\sigma]$ is $\sigma$-recognizable iff there is some $n>0$ such that $V(f)= \sum_{i=0}^{n-1}s^i(f)K$, or equivalently 
		$$s^n(f)\in \sum_{i=0}^{n-1}s^i(f)K.$$
	\end{proposition}
\begin{proof}
	If $s^n(f)\in \sum_{i=0}^{n-1}s^i(f)K$, then an application of Lemma \ref{(L)Propertiesofs} shows that   $$s^m(f)\in  \sum_{i=0}^{n-1}s^i(f)K,$$ for all $m\geq n$. Therefore, $V(f)$ is finite-dimensional, implying that $f$ is $\sigma$-recognizable, by Proposition \ref{(P)RightStableRecog}. Conversely, if $f$ is $\sigma$-recognizable, $V(f)$ must be finite-dimensional. Since $f,s(f),s^2(f),...$ generate this vector space, there exists $n>0$ such that $f,s(f),s^2(f),..., s^{n-1}(f)$ generate $V(f)$ as a right vector space. In particular, we have $s^n(f)\in \sum_{i=0}^{n-1}s^i(f)K$. 
\end{proof}
The condition $s^n(f)\in \sum_{i=0}^{n-1}s^i(f)K$ can be represented in terms of the coefficients of $f$. More precisely, let $f=\sum_{j\geq0}a_jT^j$. Then the condition $s^n(f)\in \sum_{i=0}^{n-1}s^i(f)K$ implies that there exist $c_0,...,c_{n-1}\in K$ such that
$$s^n(f)=\sum_{i=0}^{n-1}s^i(f) c_i.$$
Since $s^n(f)=\sum_{j\geq 0}a_{n+j}T^j$, we obtain
$$\sum_{j\geq 0}a_{n+j}T^j=\sum_{i=0}^{n-1}\left( \sum_{j\geq 0}a_{i+j}T^j\right)  c_i=\sum_{j\geq 0}\left(\sum_{i=0}^{n-1} a_{i+j}\sigma^j(c_i)\right) T^{j}. $$
Therefore, we see that
$$
a_{n+j}= \sum_{i=0}^{n-1} a_{i+j}\sigma^j(c_i),$$ for all $j\geq 0$, proving the following result: 
 \begin{corollary}\label{(C)recursiverecog}
 	A twisted series $f\in  K[T;\sigma]$ is $\sigma$-recognizable iff there exist $c_0,...,c_{n-1}\in K$ such  that $$
 	a_{n+j}= \sum_{i=0}^{n-1} a_{i+j}\sigma^j(c_i),$$ for all $j\geq 0$.
 \end{corollary}

\end{subsection}
	%%%%%%%%%%%%%%%%%%%%%%%%%%%%%%%%%%%%%%%%%%%%%%%%%%%%%%%%%%%%%%%%%
\begin{subsection}{The main theorem of this section} 	
In this section, we prove that every  $\sigma$-recognizable twisted series is rational, and vice versa.  We begin by some lemmas.
	\begin{lemma}\label{(L)sofproduct}
		Let $f,g\in K[[T;\sigma]]$ be arbitrary twisted series. Then, $$s(fg)=fs(g)+s(f)\sigma(g(0)).$$
In particular, if $f(0)\neq 0$, then
		$$s(f^{-1})=-f^{-1}s(f)\sigma(f^{-1}(0)).$$		
	\end{lemma}
\begin{proof}
	Given $f=\sum_{i\geq 0}a_iT^i, g=\sum_{i\geq 0}b_iT^i,$ we write 
	$$fg=\sum_{n\geq 0}\left( \sum_{i+j=n}a_i\sigma^i(b_j)\right) T^n
	\implies s(fg)=\sum_{n\geq 0}\left( \sum_{i+j=n+1}a_i\sigma^{i}(b_j)\right) T^n.
	$$
	We also have
	$$fs(g)=\left( \sum_{i\geq 0}a_{i}T^i\right) \left( \sum_{j\geq 0}b_{j+1}T^j\right)=
	\sum_{n\geq 0}\left( \sum_{i+j=n}a_{i}\sigma^{i}(b_{j+1})\right) T^n.
	$$
	It follows that 
	$$s(fg)-fs(g)=\sum_{n\geq 0}a_{n+1}\sigma^{n+1}(b_0) T^n=s(f)\sigma(b_0)=s(f)\sigma(g(0)).$$
	The second formula can be derived by setting $g=f^{-1}$ in the first formula and using the fact that $s(1)=0$. 
\end{proof}
In what follows, we identify the ring $K[[T;\sigma]]^{m\times m}$ of $m\times m$ matrices over  $K[[T;\sigma]]$ with the ring $K^{m\times m}[[T;\sigma]]$ of twisted polynomials over the ring $K^{m\times m}$. 
	\begin{lemma}\label{rationalmatrix}
		Let $M\in K[[T;\sigma]]^{m\times m}$ be a matrix with $M(0)=0$. If the entries of $M$ are rational twisted series, then so are the entries of the matrix $(I_m-M)^{-1}=\sum_{n\geq 0} M^n$, where $I_m$ denotes the identity matrix of order $m$. 
	\end{lemma}
	\begin{proof}
		We prove the lemma by induction on $n$. The assertion is trivial when $n=1$. To prove the induction step, we write $M$ as
		$$M=\begin{pmatrix}
			N&X\\
			Y& f
		\end{pmatrix}.$$
	where $N$ is a square matrix of order $m-1$, $X$ is a column vector of appropriate size, $Y$ is a row vector of appropriate size, and $f\in K[[T;\sigma]]$. We write  $(I_m-M)^{-1}$ in a similar way:
	$$(I_m-M)^{-1}=\begin{pmatrix}
		N_1&X_1\\
		Y_1& f_1
	\end{pmatrix}.$$
	The relation $(I_m-M)^{-1}(I_m-M)=1$ implies that
	$$\begin{cases}
		N_1(I_{m-1}-N)-X_1Y=I_{m-1}\\
		-N_1X+X_1(1-f)=0\\
		Y_1(I_{m-1}-N)-f_1Y=0\\
		-Y_1X+f_1(1-f)=1
	\end{cases}.
	$$
	We can solve this system of equations: The second and third equations give
	$$
	X_1=N_1X(1-f)^{-1},\,\, Y_1=f_1Y(I_{m-1}-N)^{-1},
	$$
	using which we obtain the equations
	$$\begin{cases}
		N_1(I_{m-1}-N)-N_1X(1-f)^{-1}Y=I_{m-1}\\
		-f_1Y(I_{m-1}-N)^{-1}X+f_1(1-f)=1
	\end{cases}.
	$$
	It follows that
	$$
	\begin{cases}
		N_1=\left( I_{m-1}-N-X(1-f)^{-1}Y\right)^{-1} \\
		X_1=\left( I_{m-1}-N-X(1-f)^{-1}Y\right)^{-1}X(1-f)^{-1}\\
		f_1=\left( 1-f-Y(I_{m-1}-N)^{-1}X\right)^{-1} \\
		Y_1=\left( 1-f-Y(I_{m-1}-N)^{-1}X\right)^{-1}Y(I_{m-1}-N)^{-1}
	\end{cases}.
	$$
	By induction and using the fact that sums and products of rational twisted series are rational, we see that $N_1,X_1,f_1,Y_1$ are rational, and the proof of the induction step is finished.  
	\end{proof}
Next, we show that the product of $\sigma$-recognizable twisted series is $\sigma$-recognizable, and furthermore, the inverse of an invertible $\sigma$-recognizable twisted series is $\sigma$-recognizable. 
	\begin{lemma}\label{(L)Productandinverse}
	Let $f,g\in K[[T;\sigma]]$. If  $f,g$ are $\sigma$-recognizable, then $fg$ is $\sigma$-recognizable too. If $f$ is  $\sigma$-recognizable and $f(0)\neq 0$, then $f^{-1}\in K[[T;\sigma]]$  is $\sigma$-recognizable. 
\end{lemma}
\begin{proof}
	Let $f,g\in K[[T;\sigma]]$ be  $\sigma$-recognizable. By Proposition \ref{(P)RightStableRecog}, there are stable finite-dimensional right vector subspaces $V$ and $W$ of $K[[T;\sigma]]$ such that $f\in V$ and $g\in W$. Then, the right vector space $fW+V$ contains $fg$, is finite-dimensional, and is stable since, by Lemma \ref{(L)sofproduct}, $$s(fh)=fs(h)+s(f)\sigma(h(0))\in fW+V,$$ for all $h\in W$. Now, an application of Proposition \ref{(P)RightStableRecog} proves that $fg$ is $\sigma$-recognizable.
	
	To prove the second statement, let $f\in K[[T;\sigma]]$ be $\sigma$-recognizable  and satisfy $f(0)\neq 0$.  Without loss of generality, we may assume that $f(0)= 1$. By Corollary \ref{(C)Sumofrec}, $1-f$ is  $\sigma$-recognizable. Therefore, there exists a stable finite-dimensional right vector subspace $V$ of $K[[T;\sigma]]$ such that $1-f\in V$. 
	 The right vector space $K+f^{-1}V$ is finite-dimensional and contains $f^{-1}$ since $f^{-1}=1+f^{-1}(1-f)$. Furthermore, using  Lemma \ref{(L)sofproduct}, we see that	  
	$$s(f^{-1}h)=f^{-1}s(h)+s(f^{-1})\sigma(h(0))=$$
	$$f^{-1}s(h)+f^{-1}s(1-f)\sigma(f^{-1}(0))\sigma(h(0))\in K+f^{-1}V,$$ for all $h\in V$.	Therefore, $V$ is stable. By Proposition \ref{(P)RightStableRecog},  $f^{-1}$ must be $\sigma$-recognizable, and we are done. 
\end{proof}
Now, we can prove the following result: 
	\begin{theorem}\label{(T)RecognizableRational}
		A twisted series is rational iff it is  $\sigma$-recognizable. 
	\end{theorem}
	\begin{proof}
		First, we prove that every rational twisted series is $\sigma$-recognizable. We note that that any twisted polynomial is $\sigma$-recognizable. In fact, if $P$ is a twisted polynomial, then $s^{n}(P)=0$ for all $n>\deg P$. By Proposition \ref{(P)smalleststable}, $P$ must be  $\sigma$-recognizable.  Since any rational twisted series is of the form $P^{-1}Q$, where $P,Q$ are twisted polynomials,  the result follows from Lemma \ref{(L)Productandinverse}. 
		
		Conversely, let $f=\sum_{n\geq 0}a_nT^n$ be $\sigma$-recognizable.  Suppose that $(X,A,Y)$ is a $\sigma$-linear representation of $f$.  Consider the matrix $AT\in K[[T;\sigma]]^{m\times m}$. By Lemma \ref{rationalmatrix}, all the entries of the matrix $(I_m-AT)^{-1}$ are rational. Since
		$$(I_m-AT)^{-1}=\sum_{n\geq 0}(AT)^n=\sum_{n\geq 0}\left( A\sigma(A)\cdots\sigma^{n-1}(A)\right)T^n,$$
		and $f=X(I_m-AT)^{-1}Y$, we conclude that $f$ is rational. 	
			    
	\end{proof}

\end{subsection}
	%%%%%%%%%%%%%%%%%%%%%%%%%%%%%%%%%%%%%%%%%%%%%%%%%%%%%%%%%%%%%%%%%
\begin{subsection}{Hadamard product} 
	In this part, we prove that the Hadamard product of rational twisted series is rational. 	
	The Hadamard product of twisted series $f=\sum_{i\geq 0}a_iT^i$ and $ g(T)=\sum_{i\geq 0}b_iT^i$
	is defined to be the twisted series
	$$f\odot g=\sum_{i\geq 0}a_ib_iT^i.
	$$
	The proof of the following lemma is left to the reader. 
	\begin{lemma}\label{(L)HadamardLemma}
		If $f,f_1,f_2,g,g_1,g_2\in K[[T;\sigma]]$ and $a,b,c,d\in K$, then
		$$(f_1+f_2)\odot g=f_1\odot g+f_2\odot g,$$
		$$f\odot (g_1+g_2)=f\odot g_1+f\odot g_2,$$
		$$(afb)\odot (cgd)=ac(f\odot g)bd,$$
		$$s(f\odot g)=s(f)\odot s(g).$$
	\end{lemma}	 
	Using this lemma, we prove the following result:
	\begin{theorem}
		The Hadamard product of  rational twisted series is  rational.
	\end{theorem}
	\begin{proof}
		Let $f,g\in K[[T;\sigma]]$ be rational. By Theorem \ref{(T)RecognizableRational}, $f$ and $g$ are  $\sigma$-recognizable. Therefore, by Proposition \ref{(P)RightStableRecog}, there are stable finite-dimensional right vector subspaces $V$ and $W$ of $K[[T;\sigma]]$ such that $f\in V$ and $g\in W$. Let $f_1,...,f_m$ be a basis of $V$ and $g_1,...,g_n$ be a basis of $W$. Let 		
		$V\odot  W$ be the right vector subspace of $K[[T;\sigma]]$ generated by twisted series of the form $h\odot l$ where $h\in V$ and $l\in W$.  Clearly, $f\odot g\in V\odot  W$. By Lemma \ref{(L)HadamardLemma}, we have
		$$h\odot l=\sum_i\sum_j(f_i\odot g_j)a_ib_j\in V\odot  W,$$
		where 
		$$h=\sum_if_ia_i\in V \text{ and } l=\sum_jg_jb_j\in W.$$
		Therefore, the twisted series $f_i\odot g_j$, where $i=1,...,m$ and $j=1,...,n$, generate $V\odot W$ as a right vector space. In particular, we see that  $V\odot W$  is finite-dimensional as a right vector space over $K$.  The identity (see Lemma \ref{(L)HadamardLemma})
		$$s(h\odot l)=s(h)\odot s(l),$$
		shows that $V\odot  W$  is stable. Therefore, by  Proposition \ref{(P)RightStableRecog},  $f\odot  g$ must be $\sigma$-recognizable. The result follows by applying Theorem \ref{(T)RecognizableRational}.      
	\end{proof}
	
\end{subsection}
%%%%%%%%%%%%%%%%%%%%%%%%%%%%%%%%%%%%%%%%%%%%%%%%%%%%%%%%%%%%%%%%%
\end{section} 
%%%%%%%%%%%%%%%%%%%%%%%%%%%%%%%%%%%%%%%%%%%%%%%%%%%%%%%%%%%%%%%%
%%%%%%%%%%%%%%%%%%%%%%%%%%%%%%%%%%%%%%%%%%%%%%%%%%%%%%%%%%%%%%%%
%%%%%%%%%%%%%%%%%%%%%%%%%%%%%%%%%%%%%%%%%%%%%%%%%%%%%%%%%%%%%%%%
%%%%%%%%%%%%%%%%%%%%%%%%%%%%%%%%%%%%%%%%%%%%%%%%%%%%%%%%%%%%%%%%
\begin{section}{Syntactic left ideals} 
	In this section, we assign a left ideal $I_f$ of $K[T;\sigma]$ to a given twisted series $I_f$. It is shown that a twsited series $f$ is rational iff $I_f$ is nonzero.   
	%%%%%%%%%%%%%%%%%%%%%%%%%%%%%%%%%%%%%%%%%%%%%%%%%%%%%%%%%%%%%%%%% 
	\begin{subsection}{$K[[T;\sigma]]$ as  a left $K[T;\sigma]$-module} 	
		We consider $K[T;\sigma]$ as a left vector space (over $K$). Clearly, the elements $1,T,T^2,...$ form a basis for this vector space. Given  a twisted series $f=\sum_{n\geq 0}a_nT^n\in K[[T;\sigma]]$, we define a map $L_f:K[T;\sigma]\to K[T;\sigma]$ by 
		$$L_f(\sum_{n}b_nT^n)=\sum_n  b_na_n.$$
		It is easy to see that $L_f$ is a (left) $K$-linear map, and moreover, the assignment $$f\mapsto L_f$$ establishes an isomorphism between $K[[T;\sigma]]$ and the dual of $K[T;\sigma]$ as (left) vector spaces. In fact, $L_f(T^n)=a_n$ for all $n\geq 0$. Using this assignment, one can consider $K[[T;\sigma]]$ as a left $K[T;\sigma]$-module. More precisely, 	
	 given $f\in K[[T;\sigma]]$ and $P\in K[T;\sigma]$, we  define a map $K[T;\sigma]\to K$ as follows
	$$Q\mapsto L_f(QP).$$
	It is clear that this map is (left) $K$-linear. Thus, there is a unique twisted series $P\circ f\in K[[T;\sigma]]$   such that $L_{P\circ f}(Q)=L_f(QP)$. 
	\begin{lemma}
		The operation $(P,f)\mapsto P\circ f$ turns $K[[T;\sigma]]$ into  a left $K[T;\sigma]$-module.  
	\end{lemma}
\begin{proof}
	It is easy to check that 
	$$(P_1+P_2)\circ f=P_1\circ f+P_2\circ f,$$
	$$P\circ (f_1+f_2)=P\circ f_1+P\circ f_2,$$
	$$1\circ f=f,$$
	 for all $P,P_1,P_2\in K[T;\sigma]$ and 
	 $f,f_1,f_2\in K[[T;\sigma]]$. We also have
	 $$L_{(P_1P_2)\circ f}(Q)=L_f(QP_1P_2)=L_{P_2\circ f}(QP_1)=L_{P_1\circ(P_2\circ f)}(Q),$$ showing that $(P_1P_2)\circ f=P_1\circ(P_2\circ f)$. Therefore, $(P,f)\mapsto P\circ f$ turns $K[[T;\sigma]]$ into  a left $K[T;\sigma]$-module.  
\end{proof}
	Note that if 	$f=\sum_{n}a_nT^n\in K[[T;\sigma]]$ and $P=\sum_{n}b_nT^n\in K[T;\sigma]$, then 
	\begin{equation}\label{Pof}
		P\circ f = \sum_{n\geq 0}\left( \sum_{i\geq 0}a_{n+i}\sigma^n(b_i)\right) T^n.
	\end{equation}
	In particular, we have $(P\circ f)(0)= \sum_{i}b_ia_{i}=L_f(P)$, $T\circ f=s(f)$ and $(a\circ f)=fa$ for all $a\in K$. 
	\end{subsection}
	%%%%%%%%%%%%%%%%%%%%%%%%%%%%%%%%%%%%%%%%%%%%%%%%%%%%%%%%%%%%%%%%% 
	\begin{subsection}{ The syntactic left ideal of a twisted series} 	
		Considering $K[[T;\sigma]]$ as a  left $K[T;\sigma]$-module, we define the \textit{syntactic left ideal} $I_f$ of $f\in K[[T;\sigma]]$ to be the annihilator of $f$, that is,
		$$I_f=\{P\in K[T;\sigma]\,|\, P\circ f=0\}.$$
		Clearly, $I_f$ is a left ideal of $K[T;\sigma]$. 
		Note that the map $P\mapsto P\circ f$ establishes an isomorphism $$K[T;\sigma]/I_f\to K[T;\sigma]\circ f$$ of left $K[T;\sigma]$-modules. 
		\begin{lemma}\label{(L)Largestideal}
			The syntactic left ideal of $f\in K[[T;\sigma]]$ is the largest left ideal of $K[T;\sigma]$ contained in $\ker L_f$.
		\end{lemma}
	\begin{proof}
		If $P\in I_f$, then $P\circ f=0$. It follows that $L_f(P)=(P\circ f)(0)=0$, that is, $P\in \ker L_f$. This proves that $I_f\subset \ker L_f$. To finish the proof of the lemma, we need to show that any left ideal contained in $\ker L_f$ is a subset of $I_f$. Suppose that $I\subset \ker L_f$ is a left ideal of $K[T;\sigma]$. Let  $P\in I$. We have $K[T;\sigma]P\subset I_f$, that is, $L_f(QP)=0$ for all $Q\in K[T;\sigma]$. Since  
		$$ L_{P\circ f}(Q)=L_f(QP),$$
		we see that $P\circ f=0$, hence $P\in I_f$.  
	\end{proof}

	%, and equal to the dimension of the vector space $K[T;\sigma]\circ f$.
	\end{subsection}
	%%%%%%%%%%%%%%%%%%%%%%%%%%%%%%%%%%%%%%%%%%%%%%%%%%%%%%%%%%%%%%%%%
	\begin{subsection}{The rank of a twisted series} 
			The \textit{rank} of a twisted  series $f\in K[[T;\sigma]]$ is defined to be the dimension of $K[T;\sigma]/I_f$ as a (left) vector space. Note that the rank of $f\in K[[T;\sigma]]$ is equal to the co-dimension of $I_f$ in   $K[T;\sigma]$.	
		\begin{theorem}\label{(T)Finiterank}
			(i) The rank of a rational twisted series is less than or equal to the dimension of any of its $\sigma$-linear representations. \\
			(ii) The rank of a twisted series $f$ is equal to the dimension of the right vector space $V(f)$ defined in Subsection \ref{(S)Stability}. In particular, a twisted series is rational iff it has a finite rank.   
		\end{theorem}
		\begin{proof}
	Note that by Theorem \ref{(T)RecognizableRational}, any rational twisted series is  $\sigma$-recognizable.\\
	(i) Let $(X,A,Y)$ be a $\sigma$-linear representation of a rational twisted series $f$.
	We consider the map $L:K[T;\sigma]\to K^{m\times 1}$ defined by 
	$$L\left(\sum_{n}b_nT^n \right) =\sum_{n}b_nA\sigma(A)\cdots\sigma^{n-1}(A)\sigma^{n} (Y).$$
	Clearly, $L$ is a linear map of left vector spaces. 
	I claim that $\ker L $ is a left ideal of $K[T;\sigma]$ contained in $\ker L_f$.  If $P=\sum_{n}b_nT^n\in \ker L $, then
	$$L_f(P)=\sum_{n}b_na_{n}= \sum_{n}b_nXA\sigma(A)\cdots\sigma^{n-1}(A)\sigma^{n}(Y) =$$
	$$X\left( \sum_{n}b_nA\sigma(A)\cdots\sigma^{n-1}(A)\sigma^{n}(Y)\right)=0,$$
	proving that $\ker L \subset \ker L_f$.  It remains to show that $T^mP\in \ker L $ for all $m$ and  $P=\sum_{n}b_nT^n\in \ker L $: We have
	$$T^mP=T^m\left( \sum_{n}b_nT^n\right)=\sum_{n}\sigma^m(b_n)T^{m+n}=\sum_{n\geq m}\sigma^m(b_{n-m})T^{n}.$$
	Since
	$$	\sum_{n\geq m}\sigma^m(b_{n-m})A\sigma(A)\cdots\sigma^{n-1}(A)\sigma^{n}(Y)$$
	$$=A\sigma(A)\cdots\sigma^{m-1}(A)\sigma^m\left( \sum_{n\geq m}b_{n-m}A\sigma(A)\cdots\sigma^{n-m-1}(A)\sigma^{n-m}(Y)\right) $$
	$$
	=A\sigma(A)\cdots\sigma^{m-1}(A)\sigma^m\left( \sum_{n\geq 0}b_{n}A\sigma(A)\cdots\sigma^{n-1}(A)\sigma^{n}(Y)\right)=0,
	$$
	 we see that $T^mP\in \ker L $. This finishes the proof of the claim. By Lemma \ref{(L)Largestideal}, we must have $\ker L \subset I_f$.  Since the co-dimension of $\ker L$ in $K[T;\sigma]$ is less than or equal to $m$, we see that the co-dimension of $I_f$, which is equal to the rank of $f$, is at most $m$. \\	
	(ii) Let $f=\sum_{n\geq 0}a_nT^n\in K[[T;\sigma]]$. The $K[T;\sigma]$-module isomorphism 
	$$K[T;\sigma]/I_f\to K[T;\sigma]\circ f,\,\, P\mapsto P\circ f,$$
	induces a structure of a vector space over $K$ on  $K[T;\sigma]\circ f$. This structure is given by $(a,g)\mapsto a\circ g=ga$, that is, $K[T;\sigma]\circ f$ is a right vector subspace of $K[[T;\sigma]]$.  Since  $s(g)= T\circ g$, for all $g\in K[[T;\sigma]]$, we see that $$K[T;\sigma]\circ f=\sum_{n\geq 0}s^n(f)K=V(f).$$  Since $K[T;\sigma]/I_f $ is isomorphic to $K[T;\sigma]\circ f$ (as vector spaces), we conclude that the rank of $f$ is equal to the dimension of $V(f)$ (as a right vector subspace of $K[[T;\sigma]]$). 
	\end{proof}

As an application of this theorem, we give the following result:
	\begin{corollary}\label{(C)Nonzeroideal}
		A twisted series is rational iff its syntactic left ideal is nonzero.  
	\end{corollary}
	\begin{proof}
		It is enough to note that since $K[T;\sigma]$ is a left principal ideal domain and satisfies the  Euclidean algorithm for right division, any nonzero ideal of $K[T;\sigma]$ is finite-co-dimensional in $K[T;\sigma]$.
	\end{proof}
		
	\end{subsection}

	%%%%%%%%%%%%%%%%%%%%%%%%%%%%%%%%%%%%%%%%%%%%%%%%%%%%%%%%%%%%%%%%%
\end{section} 
%%%%%%%%%%%%%%%%%%%%%%%%%%%%%%%%%%%%%%%%%%%%%%%%%%%%%%%%%%%%%%%%
%%%%%%%%%%%%%%%%%%%%%%%%%%%%%%%%%%%%%%%%%%%%%%%%%%%%%%%%%%%%%%%%
%%%%%%%%%%%%%%%%%%%%%%%%%%%%%%%%%%%%%%%%%%%%%%%%%%%%%%%%%%%%%%%%
%%%%%%%%%%%%%%%%%%%%%%%%%%%%%%%%%%%%%%%%%%%%%%%%%%%%%%%%%%%%%%%%
\begin{section}{Reduced $\sigma$-linear representations} 
	A $\sigma$-linear representation of a twisted series $f$ is called \textit{reduced} if it has  the smallest dimension among all $\sigma$-linear representations of $f$. In this section, we study some properties of reduced  $\sigma$-linear representations.
	%%%%%%%%%%%%%%%%%%%%%%%%%%%%%%%%%%%%%%%%%%%%%%%%%%%%%%%%%%%%%%%%% 
	\begin{subsection}{Ranks and reduced $\sigma$-linear representations} 	
	We begin with the following lemma. 	
	\begin{lemma}\label{(L)Rankrepresentation}
		The rank of a rational twisted series is equal to the dimension of any reduced $\sigma$-linear representation of $f$.
	\end{lemma}
	\begin{proof}
		According to Part (i) of Theorem \ref{(T)Finiterank}, the dimension of any $\sigma$-linear representation of $f$ is at least the rank of $f$. On the other hand, by Part (i) of Proposition \ref{(P)RightStableRecog}, the vector space $V(f)$ gives rise to a $\sigma$-linear representation  of dimension $\dim V(f)$. Since, by Theorem \ref{(T)Finiterank}, the rank of $f$ is equal to the dimension of $V(f)$, the result follows. 
	\end{proof}	
 In the following proposition, we give more properties of reduced $\sigma$-linear representations.
	\begin{proposition}\label{(P)Idealrepresentation}
		Let $(X,A,Y)$ be a reduced $\sigma$-linear representation of  a twisted series $f\in K[[T;\sigma]]$. Then, we have
		$$I_f=\{\sum_{n}b_nT^n\in K[T;\sigma]\,|\, \sum_{n}b_nA\sigma(A)\cdots\sigma^{n-1}(A)\sigma^{n}(Y)=0\}.$$
		Moreover, we have
		$$
		\sum_{n}KA\sigma(A)\cdots \sigma^{n-1}(A)\sigma^{n}(Y)=K^{r\times 1},$$
		where $r$ is the dimension of $(X,A,Y)$. 
	\end{proposition}
	\begin{proof}
		We consider the linear map
			$$L:K[T;\sigma]\to K^{r\times 1},\,\, L(\sum_{n}b_nT^n)= \sum_{n}b_nA\sigma(A)\cdots\sigma^{n-1}(A)\sigma^{n}(Y),$$
			introduced in the proof of  Theorem \ref{(T)Finiterank}. As seen before, $\ker L\subset  I_f$. 
		The co-dimension of $\ker L$ in $K[T;\sigma]$ is not larger than the dimension of $(X,A,Y)$. On the other hand,  the dimension of $(X,A,Y)$ is equal to the rank of $f$, see Lemma \ref{(L)Rankrepresentation}. Since the rank of $f$ is the co-dimension of $I_f$ in $K[T;\sigma]$ and $\ker L\subset I_f$, we see that $\ker L=I_f$. It follows that $L$
		is onto because the dimension of $K[T;\sigma]/\ker L=K[T;\sigma]/I_f$ is $r$. Therefore, we have
		$$
		\sum_{n\geq 0}KA\sigma(A)\cdots \sigma^{n-1}(A)\sigma^{n}(Y)=K^{r\times 1}.$$
	\end{proof}	
	
	\end{subsection}
	%%%%%%%%%%%%%%%%%%%%%%%%%%%%%%%%%%%%%%%%%%%%%%%%%%%%%%%%%%%%%%%%%
	\begin{subsection}{$\sigma$-Similarity} 
		Two  $\sigma$-linear representations  $(X_1,A_1,Y_1)$ and $(X_2,A_2,Y_2)$ are said to be $\sigma$-\textit{similar} if there is an invertible matrix $B$ over $K$ such that
		$$X_2=X_1B, \, A_2=B^{-1}A_1\sigma(B), \, Y_2=B^{-1}Y_1.$$ 
		It is easy to see that any two $\sigma$-similar $\sigma$-linear representations give rise to the same twisted series. 
		\begin{theorem}\label{(T)Reducedsimilar}
			 Let $f\in K[[T;\sigma]]$ be a rational series. Then, any two reduced $\sigma$-linear representations of $f$ are $\sigma$-similar. 
		\end{theorem}
		\begin{proof}
		Let $(X_1,A_1,Y_1)$ and $(X_2,A_2,Y_2)$ be arbitrary reduced $\sigma$-linear representations of $f$. Let $r$ be the rank of $f$. By Lemma \ref{(L)Rankrepresentation}, both $(X_1,A_1,Y_1)$ and $(X_2,A_2,Y_2)$  have dimension $r$. By Part (ii) of Proposition \ref{(P)RightStableRecog}, the right vector spaces
		$$\{X_1(I_r-A_1T)^{-1}Z|Z\in K^{r\times 1}\},$$
		and 
		$$\{X_2(I_r-A_2T)^{-1}Z|Z\in K^{r\times 1}\},$$
		are stable right vector subspaces of $K[[T;\sigma]]$. It follows that these vector subspaces are equal to $V(f)$ since $(X_1,A_1,Y_1)$ and $(X_2,A_2,Y_2)$ are reduced. As the components of each of the vectors  
		$$F_1= X_1(I_r-A_1T)^{-1} \,\text{ and } \, F_2= X_2(I_r-A_2T)^{-1},$$ form a basis for $V(f)$, we see that there exists an  invertible matrix  $B$ of order $r$ over $K$ such that $F_2=F_1B$. Now, we have
		$$F_2=F_1B\implies F_2(0)=F_1(0)B\implies X_2=X_1B,$$
		$$f=F_1Y=F_2Y_2\implies F_1Y_1=F_1BY_2\implies Y_2=B^{-1}Y_1,$$
		and
		$$s(F_1)=F_1A_1 \, \text{ and }\,  s(F_2)=F_2A_2\implies A_2= B^{-1}A_1\sigma(B),$$
		finishing the proof. 
			
		\end{proof}
		
	\end{subsection}

	%%%%%%%%%%%%%%%%%%%%%%%%%%%%%%%%%%%%%%%%%%%%%%%%%%%%%%%%%%%%%%%%%
\end{section} 
%%%%%%%%%%%%%%%%%%%%%%%%%%%%%%%%%%%%%%%%%%%%%%%%%%%%%%%%%%%%%%%%
%%%%%%%%%%%%%%%%%%%%%%%%%%%%%%%%%%%%%%%%%%%%%%%%%%%%%%%%%%%%%%%%
%%%%%%%%%%%%%%%%%%%%%%%%%%%%%%%%%%%%%%%%%%%%%%%%%%%%%%%%%%%%%%%%
%%%%%%%%%%%%%%%%%%%%%%%%%%%%%%%%%%%%%%%%%%%%%%%%%%%%%%%%%%%%%%%%
\begin{section}{Minimal fractions} 
		In this section, we introduce the concepts of minimal fraction and minimal twisted polynomial. It turns out that these two concepts are quite related. We also define the class of regular twisted series and show that they form a commutative algebra over $K$ under the Hadamard product. Throughout this section, we assume that $\sigma$ is a filed automorphism of  $K$. 	
	%%%%%%%%%%%%%%%%%%%%%%%%%%%%%%%%%%%%%%%%%%%%%%%%%%%%%%%%%%%%%%%%% 
	\begin{subsection}{Minimal fractions} 	
	 A fraction  $P^{-1}Q\in K(T;\sigma)$ is called reduced if $P$ and $Q$ are left co-prime, that is, there does not exist a nonconstant polynomial $R$ such that $$P\in RK[T;\sigma] \text{ and } Q\in RK[T;\sigma].$$ 
%	 Since $K[T;\sigma]$ is a (left and right) principal ideal domain, we see that $P$ and $Q$ are left co-prime iff 
%	  $$PK[T;\sigma] \cap  QK[T;\sigma]=\{0\}.$$ 
	 Note that every fraction $P^{-1}Q\in K(T;\sigma)$ is equal to a reduced fraction since $(RP)^{-1}(RQ)=P^{-1}Q$ for all (nonzero) $P,Q,R\in K[T;\sigma]$. A reduced fraction $P^{-1}Q\in K(T;\sigma)$ which also satisfies $P(0)=1$ will be called a \textit{minimal left fraction}. The notion of a minimal right fraction is defined in a similar way. 
	 
	 Using the Euclidean algorithm in $K[T;\sigma]$, one can show that skew polynomials $P,Q\in K[T;\sigma]$ are left co-prime iff they are right co-maximal, that is, 
	 $$P  K[T;\sigma]+QK[T;\sigma]=K[T;\sigma].$$
		\begin{lemma}\label{(L)reduced}
			Every rational twisted series  has a unique minimal left fraction. Moreover, if $P^{-1}Q$ is a minimal left fraction and $P^{-1}Q=P^{-1}_1Q_1$, where $P_1,Q_1\in K[T;\sigma]$, then $P_1=SP$ and $Q_1=SQ$ for some $S\in K[T;\sigma]$. A similar result holds for right fractions.
		\end{lemma}
		\begin{proof}
			Let $f\in K[[T;\sigma]]$ be rational. First, we prove the existence part. Consider a fraction $f=P^{-1}Q$ with $\deg P$ as small as possible. If $P(0)=0$, then $Pf=Q$ implies that $Q(0)=0$ and $f=(T^{-1}P)^{-1}(T^{-1}Q)$ would give a fraction with $\deg(T^{-1}P)<\deg P$, a contradiction. Therefore, $P(0)\neq 0$.  We may assume that $P(0)=1$ since $f=(aP)^{-1}(aQ)$ for all nonzero $a\in K$. Clearly, the fraction $f=P^{-1}Q$ is minimal. 
			
			To prove the uniqueness part, let $f=P_1^{-1}Q_1$ where $P_1,Q_1\in K[T;\sigma]$. Using the Euclidean division algorithm, we write
			$$P_1=SP+R, \text{ where } R,S\in  K[T;\sigma] \text{ and } \deg R<\deg P.$$
			It follows that $P_1f=SPf+Rf$. Since $P_1f, Pf\in K[T;\sigma]$, we see that $Rf\in K[T;\sigma]$. By the minimality of $\deg P$, we must have $R=0$. It follows that 
			$P_1=SP$ and $Q_1=SQ$. If $P_1^{-1}Q_1$ is minimal, $S$ must be a constant. It follows $P_1=P$ and $Q_1=Q$ since $P(0)=P_1(0)=1$. This proves the uniqueness of minimal left fraction. The existence and uniqueness of minimal right fractions can be proved in a similar way. 
		\end{proof}
		The unique fraction introduced in this lemma is called the \textit{minimal left fraction} of $f$.   
	\end{subsection}
%%%%%%%%%%%%%%%%%%%%%%%%%%%%%%%%%%%%%%%%%%%%%%%%%%%%%%%%%%%%%%%%% 
\begin{subsection}{Twisted reciprocal polynomials} 	
	Given $Q=\sum_{i=0}^ma_iT^i\in  K[T;\sigma]$ with $a_m\neq 0$, the \textit{left and right twisted reciprocal polynomials} of $Q$, denoted by $Q_*$ and  $Q^*$ respectively, are defined as
		$$ Q_*=\sum_{i=0}^{m}\sigma^i(a_{m-i})T^i=\sum_{i=0}^{m}T^ia_{m-i}=T^m(\sum_{i=0}^mT^{-i}a_i)\in K[T;\sigma],$$	
	$$Q^*=\sum_{i=0}^m\sigma^{i-m}(a_{m-i})T^i=(\sum_{i=0}^mT^{-i}a_i)T^m\in  K[T;\sigma].$$

\end{subsection}
	%%%%%%%%%%%%%%%%%%%%%%%%%%%%%%%%%%%%%%%%%%%%%%%%%%%%%%%%%%%%%%%%% 
	\begin{subsection}{Minimal fractions and rational twisted series}
		Let $f\in K[[T;\sigma]]$ be rational. Then, the syntactic left ideal $I_f$ of $f$ is nonzero, see Corollary \ref{(C)Nonzeroideal}. Since $ K[T;\sigma]$  is a (left and right) principal ideal domain, there must exist a unique twisted polynomial $R\in K[T;\sigma]$ with the leading coefficient equal to $1$ such that 
		$$I_f=K[T;\sigma]R.$$ 
		The	twisted polynomial $R$ is called the \textit{minimal twisted polynomial} of $f$. Note that the rank of $f$ is equal to the degree of $R$ since the rank of $f$ is equal to the dimension of the right vector space $K[T;\sigma]/I_f$. 
		
		In the following  proposition, we examine the relation between minimal right fractions and minimal twisted polynomials. 
		\begin{proposition}\label{(P)reciprocal}
			Let $f\in K[[T;\sigma]]$ be rational. Suppose that $PQ^{-1}$ is the minimal right fraction of $f$ and $R$ is the minimal twisted polynomial of $f$. Then, $T^kQ_*=R$ and $R^*=Q$, where $k=\max(0,\deg P-\deg Q+1)$. 
		\end{proposition}
		\begin{proof}
			Let $$f=\sum_{i\geq 0}a_iT^i,\, R=\sum_{i=0}^rb_iT^i,\,Q=\sum_{i=0}^qc_iT^{i},$$ where $r=\deg R$ and $q=\deg Q$.  Since $R\circ f=0$, Equation \ref{Pof} implies that 
			$$ \sum_{i=0}^r a_{n+i}\sigma^n(b_i)=0,$$
			for all $n\geq 0$. We have
			$$fR^*=\sum_{n\geq0} \left( \sum_{i+j=n}a_i\sigma^i(\sigma^{j-r}(b_{r-j}))\right) T^n.$$
			Since
			\begin{align*}
				\sum_{i+j=n+r}a_i\sigma^i(\sigma^{r-j}(b_{r-j}) &=\sum_{j=0}^r a_{n+r-j}\sigma^{n+r-j}(\sigma^{j-r}(b_{r-j}) \\
				&=\sum_{j=0}^r a_{n+r-j}\sigma^{n}(b_{r-j})\\
				&=\sum_{j=0}^r a_{n+j}\sigma^{n}(b_{j})=0,
			\end{align*} 
			for all $n\geq 0$, we see that $fR^*\in   K[T;\sigma]$, and moreover,  $\deg(fR^*)<r$. By Lemma \ref{(L)reduced}, $R^*=QS$ and $fR^*=PS$ for some $S\in  K[T;\sigma]$. In particular, we have 
			\begin{equation}\label{degpleqr}
				p\leq \deg P+\deg S=\deg(fR^*)<r,
			\end{equation}
			and 
			\begin{equation}\label{degqleqr}
				q\leq \deg Q +\deg S =\deg R^*\leq r,
			\end{equation}
			 where $p=\deg P$.
			  Since
			$$P=fQ=\sum_{n\geq 0} \left( \sum_{i+j=n}a_i\sigma^i(c_{j}) \right) T^n,$$
			we see that 
			$$\sum_{i+j=n}a_i\sigma^i(c_{j})=0,$$
			for all $n>p=\deg P$. We can rewrite this equation as
			$$\sum_{i=\max(0,n-q)}^na_i\sigma^i(c_{n-i})=0,$$
			for all $n>p=\deg P$.
			It follows that
				\begin{align*}
				\left( \sum_{i=k}^{q+k}\sigma^i(c_{q+k-i})T^i\right)\circ f &=\sum_{n\geq 0}\left(\sum_{i=k}^{q+k} a_{n+i}\sigma^n(\sigma^i(c_{q+k-i})) \right) T^n \\
				&=\sum_{n\geq 0}\left(\sum_{i=k}^{q+k} a_{n+i}\sigma^{n+i}(c_{q+k-i}) \right) T^n\\
				&=\sum_{n\geq 0}\left(\sum_{i=n+k}^{n+q+k} a_{i}\sigma^{i}(c_{n+q+k-i}) \right) T^n\\
				&=\sum_{n\geq 0}\left(\sum_{i=max(0,(n+q+k)-q)}^{n+q+k} a_{i}\sigma^{i}(c_{n+q+k-i}) \right) T^n=0,
			\end{align*} 
			since $n+q+k>p$ for all $n\geq 0$. 
			Therefore, we have $$\sum_{i=k}^{q+k}\sigma^i(c_{q+k-i})T^i\in I_f,$$ which implies that 
			$$T^kQ_*=T^k\sum_{i=0}^{q}T^ic_{q-i}=\sum_{i=k}^{q+k}\sigma^i(c_{q+k-i})T^i=S_1R,$$
			for some $S_1\in K[T;\sigma]$. 
			In particular, we have $r \leq q+k$. If $k=0$,  since $q\leq r$ by Inequality  \ref{degqleqr}, $r$ must be equal to $q$, which implies that 
			$$ T^kQ_*=\sum_{i=0}^{q}T^ic_{q-i}=R.$$
			 If $k>0$, then $r\leq q+k=p+1$ and $p<r$ (see Inequality \ref{degpleqr}) imply that $r=q+k$. Since $c_0=Q(0)=R(0)=1$, we obtain 
			$$ T^kQ_*=R.$$
			To prove $R^*=Q$, we write
			$$R^*=(T^kQ_*)^*=\left( \sum_{i=k}^{q+k}\sigma^i(c_{q+k-i})T^i\right)^*=\sum_{i=0}^q\sigma^{i-q-k}(\sigma^{i-q-k}(c_{i}))T^i=Q.$$
		\end{proof}
	Here is an application of this proposition:
	\begin{corollary}
		Let $f\in K[[T;\sigma]]$ be rational. Suppose that $PQ^{-1}$ is the minimal right fraction of $f$ and $R$ is the minimal twisted polynomial of $f$. Then, the rank of $f$ is equal to  $\max(\deg Q,\deg P+1)$. 
	\end{corollary}
	\begin{proof}
		By Proposition \ref{(P)reciprocal}, we have 
		$$\deg Q =\deg R-\max(0,\deg P-\deg Q+1)\implies \deg R=\max(\deg Q,\deg P+1).$$ 
		Since $\deg R$ is equal to the rank of $f$, the result follows. 
		
	\end{proof}
	
	\end{subsection}
%%%%%%%%%%%%%%%%%%%%%%%%%%%%%%%%%%%%%%%%%%%%%%%%%%%%%%%%%%%%%%%%% 
\begin{subsection}{Regular twisted series} 	
	By a \textit{regular twisted series}, we mean a rational twisted series in $K[[T;\sigma]]$ which has a negative degree as an element of $K(T;\sigma)$. It follows from the properties of the degree function that the set $H(\sigma)$ of all regular twisted series in $K[[T;\sigma]]$  is a subring of $K[[T;\sigma]]$. In what follows, we give a number of characterizations of regular twisted series. 
	
	\begin{lemma}\label{(L)CharacRegular}
		Let   $f\in K[[T;\sigma]]$. Then, the following are equivalent:\\
		(i) $f\in H(\sigma)$.\\
		(ii) The minimal twisted polynomial of $f$ has a nonzero constant term.\\
		(iii) For any reduced $\sigma$-linear representation $(X,A,Y)$ of $f$, the matrix $A$ is invertible. \\
		(iv) There exists a  $\sigma$-linear representation $(X,A,Y)$ of $f$ such that $A$ is invertible. 
	\end{lemma}
	\begin{proof}
	The equivalence of (i) and (ii) follows from Proposition \ref{(P)reciprocal} since the minimal twisted polynomial of $f$ has a nonzero constant term iff the number $k$ in Proposition \ref{(P)reciprocal} is zero, which in turn is equivalent to $f$ belonging to $H(\sigma)$. 
	
	 (ii)$\implies$(iii): Let $(X,A,Y)$ be an $r$-dimensional  reduced $\sigma$-linear representation of $f$. Let $$R=b_0+b_1T+\cdots+ b_{r-1}T^{r-1}+b_rT^r,$$
	 be the   minimal twisted polynomial of $f$. Note that $b_r=1$. By Proposition \ref{(P)Idealrepresentation}, we have
	 $$ b_0Y+b_1A\sigma(Y)+\cdots+b_r A\sigma(A)\cdots\sigma^{r-1}(A)\sigma^{r}(Y)=0.$$
	 Since $b_0\neq 0$, we have
	 $$ZA=0\implies ZY=0, \text{ for all } Z\in K^{1\times r}.$$
	 This fact and the equality (see Proposition \ref{(P)Idealrepresentation})
	 $$
	 \sum_{n\geq 0}KA\sigma(A)\cdots \sigma^{n-1}(A)\sigma^{n}(Y)=K^{r\times 1},$$
	 imply that $A$ is invertible.

	Since (iii)$\implies$(iv) is trivial, it only remains to show that (iv) implies (ii). Let $(X,A,Y)$ be an $m$-dimensional  $\sigma$-linear representation of $f$ such that $A$ is invertible. We consider the linear map
	$$L:K[T;\sigma]\to K^{m\times 1},\,\, L(\sum_{n}b_nT^n)= \sum_{n}b_nA\sigma(A)\cdots\sigma^{n-1}(A)\sigma^{n}(Y),$$
	introduced in the proof of  Theorem \ref{(T)Finiterank}. Clearly, the co-dimension of $\ker L$ in $K[T;\sigma]$ (as a left vector space) is at most $m$. Arguing as in the proof of Theorem \ref{(T)Finiterank}, we see that $\ker L$ is a  left ideal of $K[T;\sigma]$. It follows that there exists a (nonzero) polynomial $P\in K[T;\sigma]$ such that 
	$$\ker L=K[T;\sigma]P.$$  
	Let $$P=c_0+c_1T+\cdots+ c_{p-1}T^{p-1}+c_pT^p.$$
	 Assume that $c_0=0$. Therefore, we obtain
	$$ c_1A\sigma(Y)+\cdots+c_{p} A\sigma(A)\cdots\sigma^{p-1}(A)\sigma^{p}(Y)=0.$$
	Since $A$ is invertible, it follows that 
	$$ c_1\sigma(Y)+c_2\sigma(A)\sigma^2(Y)+\cdots+c_p \sigma(A)\cdots\sigma^{p-1}(A)\sigma^{p}(Y)=0,$$
	or equivalently, 
	$$ \sigma \left( \sigma^{-1}(c_1)Y+\sigma^{-1}(c_2)A\sigma(Y)+\cdots+\sigma^{-1}(c_p) A\sigma(A)\cdots\sigma^{p-2}(A)\sigma^{p-1}(Y)\right) =0.$$
	Since $\sigma$ is injective, we see that 
	$$ \sigma^{-1}(c_1)Y+\sigma^{-1}(c_2)A\sigma^1(Y)+\cdots+\sigma^{-1}(c_p) A\cdots\sigma^{p-2}(A)\sigma^{p-1}(Y)=0.$$
	It follows that
	$$\sigma^{-1}(c_1)+\sigma^{-1}(c_2)T+\cdots+\sigma^{-1}(c_p) T^{p-1}\in \ker L,$$
	 contradicting the choice of $P$ since $P$ has the least possible degree among the elements of $\ker L$. Therefore, $P$ has a nonzero constant term. Since 
	 $$P\in \ker L\subset I_f=K[T;\sigma]R,$$  
	 $R$ must have a nonzero constant term as well. 
	\end{proof}
		\end{subsection}
%%%%%%%%%%%%%%%%%%%%%%%%%%%%%%%%%%%%%%%%%%%%%%%%%%%%%%%%%%%%%%%%% 
\begin{subsection}{The twisted Hadamard algebra $H(\sigma)$} 
	We begin by proving that the Hadamarr product of regular twisted series is regular.  
	\begin{proposition}\label{(P)Hadamardlagbera}
 	If  $f,g\in H(\sigma)$, then $f\odot g\in H(\sigma)$.  
	\end{proposition}
	\begin{proof}
 Let $(X,A,Y)$ and $(X_1,A_1,Y_1)$ be $\sigma$-linear representations of $f$ and $g$, respectively. Since we have $\sigma(C\otimes D)=\sigma(C)\otimes \sigma(D)$ for all matrices $C$ and $D$ over $K$, one can see that $$(X\otimes X_1, A\otimes A_1,Y\otimes Y_1)$$ is a $\sigma$-linear representations of $f\odot g$. Here, $C\otimes D$ denotes the Hadamard product of matrices $C$ and $D$. Therefore, we see that $f\odot g$ has a $\sigma$-linear representation with an invertible matrix. By Lemma \ref{(L)CharacRegular}, we conclude that   $f\odot g$  is regular. 
	\end{proof}
	Since sums of regular twisted series are regular, this proposition implies that the set $H(\sigma)$ equipped with the Hadamard product is an algebra over $K$ which we will call the \textit{twisted Hadamard algebra}. The unit element of the twisted Hadamard algebra is the series $$(1-T)^{-1}=\sum_{n\geq 0}T^n.$$  Clearly, the twisted Hadamard algebra is commutative. 
\end{subsection}
	%%%%%%%%%%%%%%%%%%%%%%%%%%%%%%%%%%%%%%%%%%%%%%%%%%%%%%%%%%%%%%%%%
\end{section} 

\bibliographystyle{plain}
\bibliography{RPSbiblan}

 \end{document}